\documentclass[oneside, a4paper,11pt,reqno]{amsart}
\usepackage{xcolor}
\usepackage{ulem}
\usepackage[utf8]{inputenc}

\newtheorem{hypo}{Hypothesis}

\newtheorem{thm}[hypo]{Theorem}

\newtheorem{lem}[hypo]{Lemma}

\def\PP{\mathbb{P}}
\def\RR{\mathbb{R}}
\def\ZZ{\mathbb{Z}}

\def\EE{\mathbb{E}}
\def\NN{\mathbb{N}}

\let\BFseries\bfseries\def\bfseries{\BFseries\mathversion{bold}} 

\def \ind {{\bf 1}}

\def\dd{\mbox{d}}

\title{Persistence exponent for random walk on directed versions of $Z^2$}

\begin{document}
\author{Nadine Guillotin-Plantard} 
\address{Institut Camille Jordan, CNRS UMR 5208, Universit\'e de Lyon, Universit\'e Lyon 1, 43, Boulevard du 11 novembre 1918, 69622 Villeurbanne, France.}
\email{nadine.guillotin@univ-lyon1.fr}

\author{Fran\c{c}oise P\`ene}
\address{Universit\'e de Brest and IUF,
LMBA, UMR CNRS 6205, 29238 Brest cedex, France}
\email{francoise.pene@univ-brest.fr}

\keywords{Random walk in random scenery; local limit theorem; local time; persistence; random  walk in random environment\\
This research was supported by the french ANR project MEMEMO2}

\begin{abstract}
We study the persistence exponent for random walks in random sceneries (RWRS) with integer values and for
some special random walks in random environment in $\mathbb Z^2$
including random walks in $\mathbb Z^2$ with random orientations of the horizontal layers.
 \end{abstract}

\maketitle

\section{Introduction and main results}
Random walks in random sceneries were introduced independently by H. Kesten and F. Spitzer \cite{KS} and by A. N. Borodin \cite{Borodin}.
Let $S = (S_n)_{n \ge 0}$ be a random walk in $\mathbb{Z}$ starting at $0$,
i.e., $S_0 = 0$ and
$
X_n:= S_n-S_{n-1}, n \ge 1$ is a sequence of i.i.d. (independent identically distributed) $\mathbb{Z}$-valued random variables.
Let $\xi = (\xi_x)_{x \in \mathbb{Z}}$ be a field of i.i.d.\ $\mathbb{Z}$-valued random variables independent of $S$.
The field $\xi$ is called the random scenery.
{\it The random walk in random scenery (RWRS)} $Z := (Z_n)_{n \ge 0}$ is defined 
by setting $Z_0 := 0$ and, for $n \in \mathbb{N}^{*}$,
\begin{equation}
Z_n := \sum_{i=1}^n \xi_{S_i}.
\end{equation}
We will denote by $\mathbb{P}$ the joint law of $S$ and $\xi$. Limit theorems for RWRS have a long history, we refer to \cite{GuPo} for a complete review.

In the following, we consider the case when the common distribution of the scenery $\xi_x$  is assumed to be symmetric with a third moment and with positive variance $\sigma^2_\xi$. Concerning the random walk $(S_n)_{n\geq1}$, the distribution of $X_1$ 
is assumed to be centered and square integrable with positive variance $\sigma_X^2$. We assume without any loss of generality that neither the support of the distribution of $X_1$ nor the one of $\xi_0$ are contained in a proper subgroup of $\mathbb Z$.\footnote{If the subgroup of $\mathbb Z$ generated by the support of the distribution of $X_1$ is $m\mathbb Z$ for some $m>1$, then
we replace $(S_n)_n$ by $(S_n/m)_n$ and $(\xi_x)_x$ by $(\xi_{mx})_x$ and we observe that these changes do not affect the RWRS $Z$.
\\
If the greatest common divisor of the support of the distribution of $\xi_0$ is $\tilde m>1$, we divide the RWRS $Z$ by $\tilde m$.}

Under the previous assumptions, the following weak convergence holds in the space  of 
c\`adl\`ag real-valued functions 
defined on $[0,\infty)$,  endowed with the 
Skorokhod topology (with respect to the classical $J_1$-metric):
$$\left(n^{-\frac{1}{2}} S_{\lfloor nt\rfloor}\right)_{t\geq 0}   
\mathop{\Longrightarrow}_{n\rightarrow\infty}
^{\mathcal{L}} \left(\sigma_X\, Y(t)\right)_{t\geq 0},$$
where $Y$ is a standard real Brownian motion.
We will denote by $(L_t(x))_{x\in\mathbb{R},t\geq 0}$ a continuous version with compact support of the local time of the process $(\sigma_X\, Y(t))_{t\geq 0}$ (see \cite{marcusrosen}).
In \cite{KS}, Kesten and Spitzer proved 
the convergence in distribution of $((n^{- 3/4} Z_{[nt]})_{t\ge 0})_n$, 
to a process $\Delta=(\Delta_t)_{t\geq 0}$ defined by
$$\Delta_t : =\sigma_\xi \int_{\mathbb{R}} L_t(x) \, \dd W(x),$$
where $(W(x))_{x\ge 0}$ and $(W(-x))_{x\ge 0}$ are independent standard Brownian motions independent of $Y$. The process $\Delta$ is called Kesten-Spitzer process in the literature.
We are interested in the persistence properties
of the sum $Z_n, n\geq 1$. Our main result in this setup is the following one.
\begin{thm}
\label{theoMk}
There exists a constant $c>0$ such that for large enough $T$
\begin{equation}
\label{eqMk}
\PP\Big[ \max_{k=1,\ldots,T} Z_k \leq 1\Big]
\leq  T^{-1/4} (\log T)^{c}.
\end{equation}
If moreover $\mathbb E[e^{\xi_1}]<\infty$, then
there exist positive constants $c'$, $c''$ and $T_0$ such that 
\begin{equation}
\label{eqMklower}
 T^{-1/4} (\log T)^{-c'}\left[H^{-1}\left(c'' T^{-\frac 14}\right)\right]^{-1}\leq \PP\Big[ \max_{k=1,\ldots,T} Z_k \leq 1\Big]
\end{equation}
for every $T>T_0$, where $H$ is given by $H(t):=\mathbb E[e^{\xi_1}{\mathbf 1}_{\{e^{\xi_1}>t\}}]$
and where $ H^{-1}(x):=\inf\{t>0\ :\ H(t)<x\}$, for every
$x>0$.

In particular, if there exist $\eta>1$,  $A_1>0$ and $A_2>0$ such that 
$$\forall x>0,\quad\mathbb P(\xi_0>x)\le A_1e^{-A_2 x^\eta},$$
then there exist $c'>0$ and $T_0>0$ such that 
\begin{equation}
 T^{-1/4} e^{- c'(\log T)^{\frac 1\eta} }\leq \PP\Big[ \max_{k=1,\ldots,T} Z_k \leq 1\Big]
\end{equation}
for every $T>T_0$.

If the distribution of $\xi_1$ has compact support, then
there exist positive constants $c'$ and $T_0$ such that 
\begin{equation}
\label{eqMklower}
 T^{-1/4} (\log T)^{-c'}\leq \PP\Big[ \max_{k=1,\ldots,T} Z_k \leq 1\Big]
\end{equation}
for every $T>T_0$.

\end{thm}
The corresponding results for the continuous-time Kesten-Spitzer process $\Delta$ were obtained in \cite{BFFN}, also cf.\ \cite{Maj,MdM,castellguillotinwatbled}. The case of random walk
in random gaussian scenery was treated in \cite{AGP} with a lower bound in $ T^{-1/4} e^{- c'\sqrt{\log T}}$, coherent with our result.
We would particularly like to stress that in all these results, the scenery was supposed to be gaussian.

Now we will state an analogous result for particular models of
random walks $(M_n)_n$ in random environment on $\mathbb Z^2$ including 
random walks on $\mathbb Z^2$ with random orientation of the horizontal layers. 

To the $y$-th horizontal line, we associate the $\mathbb Z$-valued
random variable $\xi_y$, corresponding to the only authorized horizontal displacement of the walk $(M_n)_n$ on this horizontal line.
We assume that $(\xi_y)_{y\in \mathbb Z}$ is a sequence of i.i.d. random variables the distribution of which is symmetric, has a moment of order 3
and a positive variance $\sigma_\xi^{2}$.
We consider a distribution $\nu$ on $\mathbb Z$ admiting a variance and with null expectation (corresponding to the distribution of the vertical displacements when vertical displacement occur).
We fix a parameter $\delta\in(0,1)$.
We consider a random walk in random environment $M=(M_n)_n$ on $\mathbb Z^2$ starting from the origin (i.e. $M_0:=(0,0)$), moving
horizontally (with respect to $(\xi_y)_y$) with probability $\delta$ and moving vertically (with respect to $\nu$) with probability  $1-\delta$ as follows:
$$\mathbb P(M_{n+1}=(x+\xi_y,y)|M_n=(x,y))=\delta\quad\mbox{(horizontal displacement)} $$
$$\mathbb P(M_{n+1}=(x,y+z) |M_n=(x,y))=(1-\delta)\nu(\{z\})\quad\mbox{(vertical displacement)}.$$
Observe that if the $\xi_y$'s have Rademacher distribution
(i.e. takes their values in $\{-1,1\}$), then  $M$ is a walk
on $\mathbb Z^2$ with random orientations of the horizontal
layers, the $y$-th horizontal layer being oriented to the left if
$\xi_y=-1$ and to the right if $\xi_y=1$). Such models
have been considered by Matheron and de Marsilly in \cite{MdM}, their transience has been established by Campanino and P\'etritis in \cite{CP},
see also \cite{GPN} for their asymptotic behaviour
and \cite{TLL} for local limit theorem in this context. 

The process $M$ is strongly related to RWRS. Indeed it can be represented as follows
$$M_n=(\tilde Z_n,S_n)=\left(\sum_{k=1}^n\xi_{S_k}\varepsilon_k,S_n\right),\quad S_n:=\sum_{k=1}^n\tilde X_k(1-\varepsilon_k) ,$$
where $(\tilde X_k)_k$ is a sequence of i.i.d. random variables with distribution $\nu$
and $(\varepsilon_k)_k$ is a sequence of i.i.d. Bernoulli random variables with parameter $\delta$ (i.e. $\mathbb P(\varepsilon_k=1)=\delta=1-\mathbb P(\varepsilon_k=0)$).
We assume that $(\xi_y)_y$, $(\tilde X_k)_k$ and $(\varepsilon_k)_k$  
are independent. We then set $X_k:=\tilde X_k(1-\varepsilon_k)$.
As for RWRS, we assume without any loss of generality that neither the support of $\nu$ nor the one of the distribution of $\xi_0$ are contained in a proper subgroup of $\mathbb Z$. Observe that the second coordinate $S$ of $M$
is a random walk. Hence we focus our study on the first coordinate $\tilde Z$
of $M$, which is very similar to RWRS. Our second main result states that the conclusion of Theorem \ref{theoMk}
is still valid for $\tilde Z$.
\begin{thm}[Persistence of $M$ on the leftside]
\label{theoMk0}
There exists a constant $c>0$ such that for large enough $T$
\begin{equation}
\label{eqMk0}
\PP\Big[ \max_{k=1,\ldots,T} \tilde Z_k \leq 1\Big]
\leq  T^{-1/4} (\log T)^{+c}.
\end{equation}
If moreover $\mathbb E[e^{\xi_1}]<\infty$, then
there exist positive constants $c'$, $c''$ and $T_0$ such that 
\begin{equation}
\label{eqMklower0}
 T^{-1/4} (\log T)^{-c'}\left[ H^{-1}\left(c'' T^{-\frac 14}\right)\right]^{-1}\leq \PP\Big[ \max_{k=1,\ldots,T} \tilde Z_k \leq 1\Big],
\end{equation}
for every $T>T_0$. The function $H$ is defined as in Theorem \ref{theoMk}.
\end{thm}
Let us recall that $Z$ and $\tilde Z$ are stationary but non-markovian processes
with respect to the annealed distribution $\mathbb P$ and that they are markovian but non-stationary given the scenery $\xi$.

In Section \ref{preliminaryresults}, we prove some useful technical lemmas concerning the random walk $S$ as well as the random walk in random scenery
$Z$ and the analogous process $\tilde Z$. 
Section \ref{sec:rwrs} is devoted to the proof of Theorems \ref{theoMk} and \ref{theoMk0}.
\section{Preliminary results}\label{preliminaryresults}
For every $y\in\mathbb Z$ and every integer $n\ge 1$, we write $N_{n}(y)$ for the number of visits
of the walk $S$ to site $y$ before time $n$, i.e.
$$ N_n(y):=\#\{k=1,...,n\ :\ S_k=y\}.$$
Using this notation, we observe that $Z$ can be rewritten as follows:
$$Z_n=\sum_{y\in\mathbb Z}\xi_yN_n(y).$$
Analogously 
$$\tilde Z_n=\sum_{y\in\mathbb Z}\xi_y\tilde N_n(y), $$
with $\tilde N_n(y):=\#\{k=1,...,n\ :\ S_k=y\mbox{ and } \varepsilon_k=1\}$.
The behaviour of $(\tilde N_n(y))_y$ will appear to be very similar
to the behaviour of $(N_n(y))_y$, at least for our purpose.
\subsection{Preliminary results on the random walk}
We set $N_{n}^*=\sup_yN_{n}(y)$ and $R_n:=\#\{y\in\mathbb Z\ :\ N_n(y)>0\}$ for the number of sites that have been visited by the walk $S$ before time $n$.
\begin{lem}\label{Omega_n}
Let $\gamma\in(0,\frac 12)$. 
We set 
$$\Omega^{(1)}_n(\gamma):=\{N_n^*\le n^{\frac 12+\gamma},\ R_n\le  n^{\frac 12+\gamma}\}$$
There exists $C_\gamma>0$ such that
$$\mathbb P[\Omega^{(1)}_n(\gamma)]=1-O\left(\exp(-C_\gamma n^\gamma)\right).$$
\end{lem}
\begin{proof}
Due to Lemma 34 of \cite{TLL}, we know that 
there exists $c_\gamma>0$ such that 
$\mathbb P[R_n\le  n^{\frac 12+\gamma}]=1-O\left(\exp(-c_\gamma n^\gamma)\right)$. Let us prove that the argument therein can be adapted to prove the same result for $N_n^*$
instead of $R_n$.
Observe first that $a\mapsto \mathbb P[N_n^*\ge a]$ is sub-multiplicative. Indeed, let $a,b$ be two positive integers. Let us write $\tau_a:=\min\{k\ge 1,\ N_k^*=a\}$.
\begin{eqnarray*}
\mathbb P[N_n^*\ge a+b]&=&\sum_{j=1}^n\mathbb P[\tau_a=j,\ 
   N_n^*-N_j^*\ge b]\\
&\le&\sum_{j=1}^n\mathbb P[\tau_a=j,\ 
   \sup_y(N_n(y)-N_j(y))\ge b]\\
&\le&\sum_{j=1}^n\mathbb P[\tau_a=j]\mathbb P[N_{n-j}^*\ge b]\\
&\le&\mathbb P[N_n^*\ge a] \mathbb P[N_{n}^*\ge b].
\end{eqnarray*}
Hence $\mathbb P[N_n^*\ge a\, b]\le (\mathbb P[N_n^*\ge a])^b$
and so
\begin{eqnarray*}
\mathbb P[N_n^*\ge \mathbb E[N_n^*]n^\gamma]&\le&
\mathbb P[N_n^*\ge\lfloor 3\mathbb E[N_n^*]\rfloor]^{\lfloor n^\gamma/3\rfloor}\\
&\le& \left(\frac{\mathbb E[N_n^*]}{\lfloor 3\mathbb E[N_n^*]\rfloor}\right)^{\lfloor n^\gamma/3\rfloor}\le 
2^{-\lfloor n^\gamma/3\rfloor}.
\end{eqnarray*}
We conclude by using the fact that $\mathbb E[N_n^*]\sim c'\sqrt{n}$.
\end{proof}
%
\begin{lem}\label{lem:MAJOHOLDER1} 
Let $\mu\in(0,1]$ and $\gamma\in(0,1/2)$ and $\vartheta>0$ such that $\gamma  > 2(1-\mu)\vartheta$.
For any $\delta\in (0,\frac\gamma 2-(1-\mu)\vartheta)$, we have
\begin{equation}\label{MAJOHOLDER1totale}
\mathbb P\left[ \sup_{y,z\in \mathbb Z,\ 0<|y-z|<n^{\vartheta}}\frac{|N_n(y)-N_n(z)|}{{|y-z|}^{\mu}}> n^{\frac 14+\gamma}\right]=O\left(e^{- n^{\frac\delta 2}}\right).
\end{equation}
Under the assumptions of Theorem \ref{theoMk0}, we also have
\begin{equation}\label{MAJOHOLDER1totale}
\mathbb P\left[ \sup_{y,z\in \mathbb Z,\ 0<|y-z|<n^{\vartheta}}\frac{|\tilde N_n(y)-\tilde N_n(z)|}{{|y-z|}^{\mu}}> n^{\frac 14+\gamma}\right]=O\left(e^{- n^{\frac\delta 2}}\right).
\end{equation}
\end{lem}
\begin{proof}
Due to Lemma \ref{Omega_n}, it is enough to prove that
\begin{equation}\label{MAJOS_n1}
\mathbb P\Big[\sup_{k=1,...,n}|S_k|> e^{n^{\frac{\delta}{2}}}\Big]=O\left( e^{-n^{\frac{\delta}{2}} }\right)
\end{equation}
and that
\begin{equation}\label{MAJOHOLDER1}
\mathbb P\left[ \sup_{y,z\in E_n,\ 0<|y-z|<n^{\vartheta}}\frac{N_n(y)-N_n(z)}{{|y-z|}^{\mu}}> n^{\frac 14+\gamma}\right]=O\left(e^{- n^{\delta}}\right),
\end{equation}
where $E_n:=\{y\in\mathbb Z\ :\ |y|\le e^{n^{\frac{\delta}{2}}},\ N_n(y)\le n^{\frac 12+\gamma}\}$
(and the analogous estimate with $N_n(\cdot)$ replaced
by $\tilde N_n(\cdot)$ under the assumptions of Theorem
\ref{theoMk0}). 
We start with the proof of the first estimate. 
From Doob's inequality, there exists some constant $C>0$ such that
\begin{eqnarray*}
\mathbb P[\sup_{k=1,...,n}|S_k|>e^{n^{\frac{\delta}{2}}}]
&\le& C \ \mathbb E[S_n^2]\ e^{-2 n^{\frac{\delta}{2}}}= O\left(n e^{-2 n^{\frac{\delta}{2}}}\right)
\end{eqnarray*}
so \eqref{MAJOS_n1}.
Let us prove now the second estimate.
Let $\tau_j(y)$ be the $j$-th visit time of $(S_n)_n$ to $y$, that is
$$\tau_0(y):=0\quad\mbox{and}\quad\forall j\ge 0,\  \tau_{j+1}(y)=\inf\{k>\tau_j(y)\ :\ S_k=y\}$$
(resp. $\tilde\tau_0(y):=0$, $\tilde\tau_{j+1}(y)=\inf\{k>\tilde\tau_j(y)\ :\ S_k=y,\ \varepsilon_k=1\}$).
Let $y,z\in E_n$ be such that $N_n(y)-N_n(z)>0$, then there exists
$j\in\{1,...,\lfloor n^{\frac 12+\gamma}\rfloor\}$ such that $\tau_j(y)\le n<\tau_{j+1}(y)$ (observe that $\tau_{\small\lfloor n^{\frac 12+\gamma}\rfloor +1}(y)>n$). For this choice of $j$, we have
$$N_n(y)-N_n(z)\le N_{\tau_j(y)}(y)-N_{\tau_j(y)}(z).$$
Therefore
\begin{multline}
\mathbb P\left[ \sup_{y,z\in E_n,\ 0<|y-z|<n^{\vartheta}}\frac{N_n(y)-N_n(z)}{{|y-z|^\mu}}> n^{\frac 14+\gamma}\right]\\
\le\sum_{y,z\in E_n,\ 0<|y-z|<n^{\vartheta}}\sum_{j=1}^{\lfloor n^{\frac 12+\gamma}\rfloor}
    \mathbb P\left[N_{\tau_j(y)}(y)-N_{\tau_j(y)}(z)\ge {|y-z|^\mu}n^{\frac 14+\gamma}\right],\\
\le\sum_{y,z\in E_n,\ 0<|y-z|<n^{\vartheta}}\sum_{j=1}^{\lfloor n^{\frac 12+\gamma}\rfloor }
    \mathbb P\left[1+\sum_{k=1}^{j-1}(1-M_k(y,z))\ge {|y-z|^\mu}n^{\frac 14+\gamma}\right],\label{majoM_k}
\end{multline}
since $N_{\tau_j(y)}(y)-N_{\tau_j(y)}(z)=j-\sum_{k=0}^{j-1}M_k(y,z)\le j-\sum_{k=1}^{j-1}M_k(y,z)$,
where, following \cite{KS}, we write $M_k(y,z)$ for the number of visits of
$(S_n)_n$ to $z$ between its $k$-th and $(k+1)$-th visit to $y$, i.e.
$$M_k(y,z):=\sum_{\tau_{k}(y)<n\leq
\tau_{k+1}(y)}\mathbf 1_{\{S_n=z\}}.$$
Under assumptions of Theorem \ref{theoMk0}, \eqref{majoM_k}
still holds for $\tilde N_n(\cdot)$ instead of $N_n(\cdot)$
if we replace $\tau_j(y)$ by $\tilde\tau_j(y)$ and $M_k(y,z)$
by $\tilde M_k(y,z):=\sum_{\tilde\tau_{k}(y)<n\leq\tilde\tau_{k+1}(y)}\mathbf 1_{\{S_n=z,\varepsilon_n=1\}}$.

Due to the strong Markov property, $(M_k(y,z))_{k\ge 1}$ is a sequence of i.i.d.
random variables. Let us recall (see pages 13-14 in \cite{KS} for more details) that its
common law is given by
$$\mathbb P[M_k(y,z)=0]=1-p(|y-z|),\quad  \forall \ell\ge 1,\ \mathbb P[M_k(y,z)=\ell]=(1-p(|y-z|))^{\ell-1}(p(|y-z|))^2,$$
with $p(x)=p(-x)\sim \tilde c|x|^{-1}$. 
Observe also that, under the assumptions of Theorem \ref{theoMk0}, $(\tilde M_k(y,z))_{k\ge 1}$ is a sequence of i.i.d.
random variables with $\mathbb P[\tilde M_k(y,z)=0]=1-\tilde p(z-y)$ and $\mathbb P[\tilde M_k(y,z)=\ell]=\tilde p(z-y)(1-\tilde p(y-z))^{\ell-1}\tilde p(y-z)$ if $\ell\ge 1$
where $\tilde p(x)$ denotes the probability that $(S_n,\varepsilon_n)_{n\ge 1}$ visits $(x,1)$ before $(0,1)$. Observe that
\footnote{The fact that $(S_n,\varepsilon_n)_{n\ge 1}$ visits $(x,1)$ before $(0,1)$ means that $S$ visits $0$ several times (let us say $k$ times, with $k\ge 0$) before its first visit at $x$ but that $\epsilon_n=0$ at each of these visits to $0$ (this happens with probability $(1-\delta-p(x))^k$), that $S$ goes to $x$ before coming back to $0$ (this happens with probability $p(x)$) and finally
that, starting from $S=x$, $(S_n,\varepsilon_n)_{n\geq 1}$ visits $(x,1)$ before $(0,1)$
(this happens with probability $1-\tilde p(-x)$).}
$$\tilde p(x)=\sum_{k\ge 0}(1-\delta-p(x))^{k}p(x)(1-\tilde p(-x))=\frac{p(x)}{\delta+p(x)}(1-\tilde p(-x)).$$
Iterating this formula we obtain that $\tilde p(x)=
\frac {p(x)}{\delta+p(x)}\left(1-\frac{p(x)}{\delta+p(x)}(1-\tilde p(x))\right)=\frac{p(x)(\delta+p(x)\tilde p(x))}{(\delta+p(x))^2}$ which leads to
$$\frac{p(x)}{2+\delta}\le\tilde p(x)=\frac{p(x)}{\delta+2p(x)}\le \frac{p(x)}\delta. $$
There exists $C_0>1$ such that
\begin{equation}\label{bornesp}
\forall x\ne 0,\quad C_0^{-1}|x|^{-1}\le p(x)\le C_0|x|^{-1}
\end{equation}
and
\begin{equation}\label{bornesp0}
\forall x\ne 0,\quad C_0^{-1}|x|^{-1}\le \tilde p(x)\le C_0|x|^{-1}.
\end{equation}
Observe that $M_1(y,z)$ has expectation 1
and admits exponential moment of every order:
$$\forall t>0,\quad G_{|y-z|}(t):=\mathbb E\left[e^{t(1-M_1(0,|y-z]))}\right]=\frac{(1-p(|y-z|))e^t-1+2p(|y-z|)}{1-(1-p(|y-z|))e^{-t}}.$$
Hence, for every positive integer $J\le n^{\frac 12+\gamma}$, due to the Markov inequality, 
we obtain that for every $t>0$,
$$\mathbb P\left[1+\sum_{k=1}^{J}(1-M_k(y,z))\ge{|y-z|^\mu} n^{\frac 14+\gamma}\right]=
\mathbb P\left[\exp\left(t+t\sum_{k=1}^{J}(1-M_k(y,z))\right)\ge \exp\left(t{|y-z|^\mu}\, n^{\frac 14+\gamma}\right)\right]$$
\begin{eqnarray*}
&\le&\exp\left(-t\,{|y-z|^\mu} n^{\frac 14+\gamma}\right)\mathbb E\left[\exp\left(t+t\,\sum_{k=1}^{J}(1-M_k(y,z))\right)\right]\\
&\le& \exp\left(-t\, {|y-z|^\mu}n^{\frac 14+\gamma}\right)(G_{|y-z|}(t))^Je^t\\
&\le& \exp\left(-t\, {|y-z|^\mu}n^{\frac 14+\gamma}\right)\left(\frac{(1-C_0^{-1}|y-z|^{-1})e^t-1+2C_0^{-1}|y-z|^{-1}}{1-(1-C_0^{-1}|y-z|^{-1})e^{-t}}\right)^Je^t\\
&\le& \exp\left(-t\, {|y-z|^\mu}n^{\frac 14+\gamma}\right)\left(\frac{(1-C_0^{-1}|y-z|^{-1})e^t-1+2C_0^{-1}|y-z|^{-1}}{1-(1-C_0^{-1}|y-z|^{-1})e^{-t}}\right)^{n^{\frac 12+\gamma}}e^t\\
\end{eqnarray*}
since the function $f:p\mapsto \frac{(1-p)e^t-1+2p}{1-(1-p)e^{-t}}$ is decreasing on $(0,1)$ such that $f(0)=e^t$ and $f(1)=1$.
Now using the Taylor expansion of $e^t$ at $0$, we observe that
$$ \frac{(1-p)e^t-1+2p}{1-(1-p)e^{-t}}=\frac{1+\frac qpt+\frac qp\frac{t^2}2+\frac qp O(t^3)}{1+\frac qpt-\frac qp\frac{t^2}2+\frac qp O(t^3)}$$
with $p=C_0^{-1}|y-z|^{-1}$ and $q=1-p$ where $O(t^3)$ is uniform in $p$.
Taking $t=p\, n^{-\frac 14-\frac\gamma 2}$, we obtain
$$ \frac{(1-p)e^t-1+2p}{1-(1-p)e^{-t}}=\frac{1+q\, n^{-\frac 14-\frac\gamma 2}+qp\frac{n^{-\frac 12-\gamma}}2+p^2 O(n^{-\frac 34-\frac{3\gamma} 2})}{1+q\, n^{-\frac 14-\frac\gamma 2}-qp\frac{n^{-\frac 12-\gamma}}2+p^2 O(n^{-\frac 34-\frac{3\gamma} 2})}=1+qp\, n^{-\frac 12-\gamma}+O(n^{-\frac 34-\frac{3\gamma} 2}).$$
and so
$$\mathbb P\left[1+\sum_{k=1}^{J}(1-M_k(y,z))\ge {|y-z|^\mu} n^{\frac 14+\gamma}\right]
=O\left(e^{-\, C_0^{-1}|y-z|^{\mu-1}n^{\frac\gamma 2}}\right)=O\left(e^{- C_0^{-1} n^{-(1-\mu)\vartheta+\frac\gamma 2}}\right).$$
Taking $\delta\in (0,\frac\gamma 2-(1-\mu)\vartheta)$
and combining this with \eqref{majoM_k}, 
we deduce \eqref{MAJOHOLDER1} and the analogous estimate for
$\tilde N_n(\cdot)$ instead of $N_n(\cdot)$ under the assumptions of Theorem \ref{theoMk0} (replacing $M_k$ by $\tilde M_k$
and  $p(\cdot)$ by $\tilde p(\cdot)$ in the above argument).
\end{proof}
\subsection{A conditional local limit Theorem for the RWRS}
Let $\varphi_\xi$ be the characteristic function of $\xi_1$.
Since $\xi_1$ takes integer values, $e^{2i\pi\xi_1}=1$ a.s. and
so $\varphi_\xi(u)=1$ for every $u\in 2\pi\mathbb Z$. Let us consider the positive integer $d$ such that $d\, \{u\ :\ |\varphi_\xi(u)|=1\}=2\pi\mathbb Z$. Another characterization of $d$ is that
it is the positive generator of the subgroup of $\mathbb Z$ generated by the $b-c$, with $b$ and $c$ in the support of the distribution of $\xi_1$ (i.e. by the support of the distribution of $\xi_0-\xi_1$).
Since the support of $\xi_1$ is not contained in a proper subgroup of $\mathbb Z$, we also have $d=\inf\{n\ge 1\, :\ e^{2i\pi n\xi_1/d}=1\ \text{a.s.}\}$.
Observe that $e^{\frac{2i\pi}d\xi_1}$ is almost surely constant
and so $(e^{\frac{2i\pi}d\xi_1})^2=\varphi_\xi\left(\frac {2\pi}d\right)^2
=e^{\frac{2i\pi}d(\xi_0+\xi_1)}$ almost surely.
Since the distribution of $\xi_1$ is symmetric, $\mathbb P(\xi_0+\xi_1=0)>0$ and so $(e^{\frac{2i\pi}d\xi_1})^2=1$ almost surely. Hence either $d{}=1$ (and  $e^{\frac{2i\pi}d\xi_1}=1$ a.s.) or $d{}= 2$ (and  $e^{\frac{2i\pi}d\xi_1}=-1$ a.s.).
The following lemma relates the conditional probability
$\mathbb P[\left. Z_{n}=0\, \right|S]$ to the self-intersection
local time $V_n$ of the random walk $S$ up to time $n$.
Let us recall that $V_n$ is given by
$$V_n:=\sum_{k,\ell=1}^n\ind_{\{S_k=S_\ell\}}=\sum_{k,\ell=1}^n\sum_{y\in\mathbb Z}\ind_{\{S_k=S_\ell=y\}}=\sum_y(N_n(y))^2.$$
Under the assumptions ot Theorem \ref{theoMk0}, we set 
$\tilde V_n:=\sum_y(\tilde N_n(y))^2$.
\begin{lem}\label{TLLcond3}
Let $\gamma\in(0,1/48)$.
There exists a sequence of $S$-measurable sets $(\Omega_{n}^{(0)}(\gamma))_n$, an integer $n_0>0$ and a positive
constant $c$ such that $\mathbb P(\Omega_n^{(0)}(\gamma))=1+O\left(e^{-n^{\frac\delta 2}}\right)$ for any $\delta<\frac\gamma 2$ and
such that, for every $n\ge n_0$ such that $n\in d{}\mathbb N$, the following inequalities hold on $\Omega_{n}^{(0)}(\gamma)$:
$$\mathbb P[\left. Z_{n}=0\, \right|S]\ge \frac{c}{\sqrt{V_{n}}},$$
$$n^{\frac 12-\gamma }\le N_{n}^*\le n^{\frac 12+\gamma },\quad
R_{n}\le n^{\frac 12+\gamma }  \quad\mbox{and}\quad     n^{\frac 32-\gamma }\le V_{n}\le n^{\frac 32+\gamma }. $$
Under the assumptions of Theorem \ref{theoMk0},
there exists a sequence of $(S,(\varepsilon_k)_k)$-measurable sets $(\tilde\Omega_{n}^{(0)}(\gamma))_n$, an integer $n_0>0$ and a positive
constant $c$ such that $\mathbb P(\Omega_n^{(0)}(\gamma))=1+O\left(e^{-n^{\frac\delta 2}}\right)$ for any $\delta<\frac\gamma 2$ and
such that, for every $n\ge n_0$ such that the following inequalities hold on $\tilde\Omega_{n}^{(0)}(\gamma)$:
$$\mathbb P\left[\left.\tilde Z_{n}=0\, \right|(S,(\varepsilon_k)_k)\right]\ge \frac{c}{\sqrt{\tilde V_{n}}}\ind_{\left\{\sum_{k=1}^n\varepsilon_k\in d{}\mathbb N^*\right\}},$$
$$n^{\frac 12-\gamma }\le \sup \tilde N_n(\mathbb Z)\le N_{n}^*\le n^{\frac 12+\gamma },\quad
R_{n}\le n^{\frac 12+\gamma }  \quad\mbox{and}\quad     n^{\frac 32-\gamma }\le \tilde V_{n}\le n^{\frac 32+\gamma }. $$
\end{lem}

\noindent {\bf Remark:} If we assume that $\varphi_\xi$ is non negative (in this case $\mathbb P(\xi_1=0)>0$), then there exists $c>0$ such that for every $n\ge 1$
$$\mathbb P[\left. Z_{n}=0\, \right|S]\ge \frac{c}{\sqrt{V_{n}}}.$$
Indeed, observe that
\begin{equation}\label{Fourier}
\mathbb P[Z_{n}=0|S]=\frac 1{2\pi}\int_{-\pi}^\pi\mathbb E\left[e^{itZ_{n}}|S\right]\, dt
=\frac 1{2\pi}\int_{-\pi}^\pi\prod_{y}
     \varphi_\xi(tN_{n}(y))\, dt.
\end{equation}
Remark that for every $y\in\mathbb Z$, $N_n(y)\le \sqrt{V_n}$. We know that $\varphi_\xi(t)-1\sim-\frac{\sigma^2_\xi}2\, t^2$.
Let $\beta>0$ be such that, for every real number $u$ satisfying $|u|<\beta$,
we have $\varphi_\xi(u)\ge e^{-{\sigma^2_\xi}u^2}$. Since
$\varphi_\xi$ is non negative, we have
\begin{eqnarray*}
\mathbb P[Z_{n}=0|S]&\ge&
\frac 1{2\pi}
\int_{-\beta/\sqrt{V_n}}^{\beta/\sqrt{V_n}}\prod_{y}  \left[\varphi_\xi(tN_{n}(y))\right]\, dt\\
&\ge &\frac 1{2\pi}
\int_{-\beta/\sqrt{V_n}}^{\beta/\sqrt{V_n}}\prod_{y} e^{-{\sigma^2_\xi}t^2(N_{n}(y))^2}\, dt\\
&=&\frac 1{2\pi}
\int_{-\beta/\sqrt{V_n}}^{\beta/\sqrt{V_n}}e^{-{\sigma^2_\xi}t^2V_{n}}\, dt\\
&\ge&\frac{1}{2\pi\sigma_\xi\sqrt{V_{n}}}\int_{|u|<\sigma_\xi\beta}e^{-u^2} \, du.
\end{eqnarray*}
The proof of Lemma \ref{TLLcond3} is based on the same idea.
The fact that $\varphi_\xi$ can take negative values 
complicates the proof.
\begin{proof}[Proof of Lemma \ref{TLLcond3}]
We have
\begin{equation}
\mathbb P[Z_{n}=0|S]=\frac 1{2\pi}\int_{-\pi}^\pi\mathbb E\left[e^{itZ_{n}}|S\right]\, dt
=\frac 1{2\pi}\int_{-\pi}^\pi\prod_{y}
     \varphi_\xi(tN_{n}(y))\, dt.
\end{equation}
Observe that $e^{2i\pi\xi_1/d}=\mathbb E[e^{2i\pi\xi_1/d}]$ almost surely
and so $\mathbb E[e^{2i\pi\xi_1/d}]^d=\mathbb E[e^{2i\pi\xi_1}]=1$.
Hence, for any integer $m\ge 0$ and any $u\in\mathbb R$, we have
$$\varphi_\xi\left(\frac{2m\pi}d+u\right)=
\left({\varphi_\xi\left(\frac{2\pi}d\right)}\right)^m\varphi_\xi(u) $$
and so
\begin{eqnarray}
\mathbb P[Z_{n}=0|S]&=&\frac 1{2\pi}\int_{-\pi}^\pi\mathbb E\left[e^{itZ_{n}}|S\right]\, dt\nonumber\\
&=&\frac 1{2\pi}\sum_{k=0}^{d-1}\int_{-\pi/d}^{\pi/d}\prod_{y}
     \left[\left({\varphi_\xi\left(\frac{2\pi}d\right)}\right)^{kN_{n}(y)}
\varphi_\xi(tN_{n}(y))\right]\, dt\nonumber\\
&=&\frac 1{2\pi}\sum_{k=0}^{d-1}
\left({\varphi_\xi\left(\frac{2\pi}d\right)}\right)^{kn}
\int_{-\pi/d}^{\pi/d}\prod_{y}  \left[\varphi_\xi(tN_{n}(y))\right]\, dt\nonumber\\
&=&\frac d{2\pi}
\int_{-\pi/d}^{\pi/d}\prod_{y}  \left[\varphi_\xi(tN_{n}(y))\right]\ind_{\{n\in d{}\mathbb N^*\}}\, dt\label{Fourier_bis}.
\end{eqnarray}
Under the assumptions of Theorem \ref{theoMk0}, proceeding analogously we obtain
\begin{equation}
\mathbb P[\tilde Z_{n}=0|(S,(\varepsilon_k)_k)]
=\frac d{2\pi}
\int_{-\pi/d}^{\pi/d}\prod_{y}  \left[\varphi_\xi(t\tilde N_{n}(y))\right]\ind_{\{\sum_{k=1}^n\varepsilon_k\in d{}\mathbb N^*\}}\, dt\label{Fourier_bis0},
\end{equation}
since $\sum_{y\in\mathbb Z}\tilde N_n(y)=\sum_{k=1}^n\varepsilon_k$.
We know that $\varphi_\xi(t)-1\sim-\frac{\sigma^2_\xi}2\, t^2$.
Let $\beta>0$ be such that, for every real number $u$ satisfying $|u|<\beta$,
we have $ e^{-{\sigma^2_\xi}u^2}\le\varphi_\xi(u)\le e^{-\frac{\sigma^2_\xi}4u^2}$ (observe that the fact that the distribution of $\xi$
is symmetric implies that $\varphi_\xi$ takes real values). 
Using the fact that $N_{n}(y)\le N_{n}^*\le \sqrt{V_{n}}$, we have
\begin{eqnarray*}
\frac d{2\pi}
\int_{-\beta/N_{n}^*}^{\beta/N_{n}^*}\prod_{y}  \left[\varphi_\xi(tN_{n}(y))\right]\, dt
&\ge &\frac d{2\pi}
\int_{-\beta/\sqrt{V_{n}}}^{\beta/\sqrt{V_{n}}}\prod_{y} e^{-{\sigma^2_\xi}t^2(N_{n}(y))^2}\, dt\\
&=&\frac d{2\pi}
\int_{-\beta/\sqrt{V_{n}}}^{\beta/\sqrt{V_{n}}}e^{-{\sigma^2_\xi}t^2V_{n}}\, dt\\
&\ge&\frac{d}{2\pi\sigma_\xi\sqrt{V_{n}}}\int_{|u|<\sigma_\xi\beta}e^{-u^2} \, du.
\end{eqnarray*}
This gives
\begin{equation}\label{INT1}
\frac d{2\pi}
\int_{-\beta/N_{n}^*}^{\beta/N_{n}^*}\prod_y  \varphi_\xi(tN_{n}(y))\, dt
   \ge \frac c{\sqrt{V_{n}}},
\end{equation}
for some positive constant $c$,
and analogously
\begin{equation}\label{INT1_bis}
\frac d{2\pi}
\int_{-\beta/\tilde N_{n}^*}^{\beta/\tilde N_{n}^*}\prod_y  \varphi_\xi(tN_{n}(y))\, dt
   \ge \frac c{\sqrt{\tilde V_{n}}},
\end{equation}
under the assumptions of Theorem \ref{theoMk0} if $\tilde V_n\ne 0$.

Let $\Omega_n(\gamma)$ be the set defined by
$$\Omega_n(\gamma)=\left\{R_{n}\le n^{\frac 12+\gamma }, N_n^*\le n^{\frac 12+\gamma },\  \sup_{y\ne z; |y-z|\leq n}\frac{|N_n(y)-N_n(z)|}{{|y-z|}}\le n^{\frac 14+\gamma}
\right\}.$$
Due to Lemmas \ref{Omega_n} and \ref{lem:MAJOHOLDER1} (applied with $\mu=1$
and $\vartheta=1$), $\mathbb P(\Omega_n(\gamma))=1+O\left(e^{-n^{\frac\delta 2}}\right)$ for any $\delta<\frac\gamma 2$.
On $\Omega_{n}(\gamma)$, due to the Cauchy-Schwartz inequality, we have 
$n=\sum_yN_n(y)\mathbf 1_{\{N_n(y)>0\}}\le\left(V_n\sum_{y}\mathbf 1_{\{N_n(y)>0\}}\right)^{\frac 12}\le \sqrt{R_n\, V_n}$  and so 
$V_{n}\ge n^{\frac 32-\gamma}$. Observe also that $V_n\le N_n^*\sum_y N_n(y)=n\,
N_n^*\le n^{\frac 32+\gamma}$.
Moreover $n=\sum_yN_{n}(y)\le R_{n}N_{n}^*$. Hence 
$N_{n}^*\ge n^{\frac 12-\gamma}$. This gives the three last inequalities in the first case.

Under the assumptions of Theorem \ref{theoMk0}, we set analogously
$$\tilde\Omega_n(\gamma)=\left\{R_{n}\le n^{\frac 12+\frac{3\gamma} 4}, N_n^*\le n^{\frac 12+\frac{3\gamma} 4},\  \sup_{y\ne z; |y-z|\leq n}\frac{|\tilde N_n(y)-\tilde N_n(z)|}{{|y-z|}}\le n^{\frac 14+\gamma},\
   \sum_{k=1}^n\varepsilon_k> \frac{n(1-\gamma)\delta}4\right\}.$$
If $n$ is large enough, on $\tilde\Omega_n(\gamma)$, we also obtain that
$n^{\frac 32-\gamma}\le \tilde V_{n}\le n^{\frac 32+\gamma}$
and $\sup \tilde N_n(\mathbb Z)\ge n^{\frac 12-\gamma}$
using the same arguments as above and the fact that $\sum_y\tilde N_n(y)=\sum_{k=1}^n\varepsilon_k$.

To end the proof of the lemma, it remains to prove the first inequality. Due to \eqref{Fourier_bis}
and \eqref{INT1}, it remains to prove that there exists $n_1>0$
such that, for every $n\ge n_1$, on $\Omega_{n}(\gamma)$,
we have
\begin{equation}\label{INT1b_bis}
\int_{\frac\beta{N_{n}^*} \le |t|\le\frac \pi d}
     \prod_y  |\varphi_\xi(tN_{n}(y))|\, dt
\le\int_{\beta n^{-\frac 12-\gamma} \le |t|\le\frac \pi d}
     \prod_y  |\varphi_\xi(tN_{n}(y)) |\, dt
   \le \frac{c}{2\sqrt{V_{n}}},
\end{equation}
and the analogous inequality obtained by replacing  $N_n(\cdot)$ by $\tilde N_n(\cdot)$ and $V_n$ by $\tilde V_n$, under the assumptions of Theorem \ref{theoMk0}. 
To this end, we will use elements of the proof of \cite{TLL} and more precisely the
proofs of Propositions 9 and 10 therein. 

We fix $\varepsilon>3\gamma$ such that $3\gamma+3\varepsilon<\frac 14$ (this is possible since
$\gamma<\frac 1{48}$).
We first follow the proof of Proposition 9 in \cite{TLL} (here $\varepsilon_0=\beta$) and more
precisely of Lemma 14 therein.
Let $y_1\in\mathbb Z$ be such that $N_n(y_1)= N_n^*$ and set $y_0:=\min\{y\ge y_1\, : \, 
N_n(y)\le\frac\beta 2 n^{\frac 12-\varepsilon}\}$. On $\Omega_n(\gamma)$, for $n$ large enough, $y_0>y_1$ (since $\varepsilon> \gamma$) and so $N_n(y_0-1)>\frac\beta 2 n^{\frac 12-\varepsilon}\ge N_n(y_0)$.
Moreover, still on $\Omega_n(\gamma)$,
$N_n(y_0-1)-N_n(y_0)\le n^{\frac 14+\gamma}$ which is smaller than
$\frac\beta 4 n^{\frac 12-\varepsilon}$ for $n$ large enough so that
$\frac\beta 4 n^{\frac 12-\varepsilon}\le N_n(y_0)\le \frac\beta 2 n^{\frac 12-\varepsilon}$.
Now, on $\Omega_n(\gamma)$, for every $z\in\mathbb Z$ such that $|y_0-z|\le e_n:=\frac{\beta}{10}n^{\frac 14-\varepsilon-\gamma}$, then
$$|N_n(z)-N_n(y_0)|\le |y_0-z|n^{\frac 14+\gamma}\le\frac\beta{10}n^{\frac 12-\varepsilon}$$
and so
$$\frac\beta{10}n^{\frac 12-\varepsilon}<N_n(z)<\beta n^{\frac 12-\varepsilon} ,$$
hence $|tN_n(z)|\le \beta$ if $n^{-\frac 12-\gamma}<|t|<n^{-\frac 12+\varepsilon}$ and so, on $\Omega_n(\gamma)$,
\begin{eqnarray*}
\prod_{y\in\mathbb Z}|\varphi_\xi(tN_n(y))|&\le& 
   \exp\left( - \frac{\sigma^2_\xi}4 t^2\sum_{z=y_0-e_n}^{y_0+e_n}(N_n(z))^2  \right)\\
&\le&\exp\left( - \frac{\sigma^2_\xi}4 n^{-1-2\gamma} 2e_n\frac{\beta^2}{100}n^{1-2\varepsilon}  \right)\\
&\le&\exp\left( - \frac{\sigma^2_\xi}2  n^{\frac 14-3\gamma-3\varepsilon} \frac{\beta^3}{10^3}  \right).
\end{eqnarray*}
Hence we have proved that, for $n$ large enough, on $\Omega_n(\gamma)$,
\begin{equation}\label{autourde1/2}
\int_{\beta n^{-\frac 12-\gamma} \le |t|\le n^{-\frac 12+\varepsilon}}
     \prod_y  | \varphi_\xi(tN_{n}(y))|\, dt
   \le \frac{c}{4\sqrt{V_{n}}}
\end{equation}
since $3\gamma+3\varepsilon<\frac 14$.

Under the assumptions of Theorem \ref{theoMk0}, the same argument gives
\begin{equation}\label{autourde1/20}
\int_{\beta n^{-\frac 12-\gamma} \le |t|\le n^{-\frac 12+\varepsilon}}
     \prod_y  | \varphi_\xi(t\tilde N_{n}(y))|\, dt
   \le \frac{c}{4\sqrt{\tilde V_{n}}}
\end{equation}
Now, Lemma 15 of \cite{TLL} still holds with our set $\Omega_n(\gamma)$ since the proof only uses
the fact that $N_n^*\le n^{\frac 12+\gamma}$ and that $R_n\le n^{\frac 12+\gamma}$.
Due to this remark and using the notations and results contained in Section 2.8 of \cite{TLL},
we take
for $\Omega_{n}^{(0)}(\gamma)$ the subset of $\Omega_{n}(\gamma)\cap \mathcal D_{n}$
on which $\#\{z\ :\ N_{n}(z)\in\mathcal I\}\ge {n}^{\frac 12-2\gamma}/4$ (with $\mathcal D_n$ and $\mathcal I$ being defined in Section 2.8 of \cite{TLL} applied with $\alpha=2$).
Since $\gamma<\frac 18$ and $3\gamma<\varepsilon<\frac 12$,
we obtain that $\mathbb P(\Omega_{n}(\gamma)\setminus\Omega_{n}^{(0)}(\gamma))=o(e^{-cn})$ for some $c>0$\footnote{Indeed, using the notations $\mathcal D_n$, $\mathcal E_n$ and $\mathcal I$ of Section 2.8 of \cite{TLL}, $\mathbb P(\mathcal D_n)=1-o(e^{-cn})$; moreover
following the proof of Lemma 15 of \cite{TLL} we obtain that $\Omega_n(\gamma)\cap\mathcal D_n\subset \mathcal E_n$, and finally, due to the remark following Lemma 17
of \cite{TLL}, $p_2(n):=\mathbb P(\mathcal E_n,\ \#\{z\ :\ N_{n}(z)\in\mathcal I\}< {n}^{\frac 12-2\gamma}/4)=o(e^{-cn})$. Therefore
$\mathbb P(\Omega_n(\gamma)\setminus\Omega_n^{(0)}(\gamma))\le\mathbb P(\Omega_n(\gamma)\setminus \mathcal D_n)+p_2(n)=o(e^{-cn})$.}. Moreover, there exists an integer $n_2$ such that if $n\ge n_2$, on $\Omega_{n}^{(0)}(\gamma)$ we have
\begin{equation}\label{INT4}
\forall t\in[n^{-\frac 12+\varepsilon},\frac \pi d],\quad
\prod_y  |\varphi_\xi(tN_{n}(y))|\le \exp(-n^{\gamma}) \leq  \frac{c}{4\sqrt{V_{n}}}
\end{equation}
(see the lines before the proof of Lemma 17 of \cite{TLL}).
The same argument (with the flat peaks instead of the peaks 
as explained in Section 5.4 of \cite{TLL} gives also (for every $n$ large enough)
\begin{equation}\label{INT40}
\forall t\in[n^{-\frac 12+\varepsilon},\frac \pi d],\quad
\prod_y  |\varphi_\xi(t\tilde N_{n}(y))|\le \exp(-n^{\gamma}) \leq  \frac{c}{4\sqrt{\tilde V_{n}}}
\end{equation}
on some set $\tilde\Omega_{n}^{(0)}(\gamma)$ such that 
$\mathbb P(\tilde\Omega_n(\gamma)\setminus\tilde\Omega_n^{(0)}(\gamma))=o(e^{-cn})$.
\end{proof}
\subsection{A conditional Berry-Esseen bound for RWRS}
Let $\Phi$ be the distribution function of the standard gaussian distribution, i.e. 
\begin{equation}\label{eqn:defnbarphi}
\forall u\in \RR,\quad\Phi(u)=\frac{1}{\sqrt{2\pi}} \int_{-\infty}^u e^{-\frac{x^2}{2} }\dd x.
\end{equation}
For $p\geq 1$, let us define the $p$-fold self-intersection local time of the random walk up to time $n$
$$Q_n^{(p)} := \sum_{y\in\ZZ} N_n(y)^p.$$
Under the assumptions of Theorem \ref{theoMk0}, we define
$$\tilde Q_n^{(p)}  := \sum_{y\in\ZZ} \tilde N_n(y)^p.$$
\begin{lem}\label{berry}
There exists a positive constant $C$ such that for every $n\ge 1$ 
$$\sup_{x\in\RR} \left| \PP\left[\frac{Z_n}{\sigma_{\xi}\sqrt{V_n}} \leq x \Big| S\right] - \Phi(x)\right|\leq C \frac{\EE[|\xi_0|^3]}{\EE[|\xi_0|^2]^{3/2}} \frac{Q_n^{(3)} }{ V_n^{3/2}}$$
and such that, under the assumptions of Theorem \ref{theoMk0},
$$\sup_{x\in\RR} \left| \PP\left[\frac{\tilde Z_n}{\sigma_{\xi}\sqrt{\tilde V_n}} \leq x \Big| (S,(\varepsilon_k)_k)\right] - \Phi(x)\right|\leq C \frac{\EE[|\xi_0|^3]}{\EE[|\xi_0|^2]^{3/2}} \frac{\tilde Q_n^{(3)} }{ \tilde V_n^{3/2}}.$$
\end{lem}
\begin{proof}
This result directly follows from Berry-Esseen theorem since conditionally on the random walk, $Z_n$ (resp. $\tilde Z_n$) is the sum of centered,  independent random variables $\xi_yN_n(y)$ (resp. $\xi_y\tilde N_n(y)$ under the assumptions of Theorem \ref{theoMk0}).
\end{proof}

\section{Proof of Theorems \ref{theoMk} and \ref{theoMk0}} \label{sec:rwrs}
\subsection{Relation to exponential functionals} \label{sec:expfunctionals}
The main idea is to relate the persistence probability to the exponential functional $\sum_{\ell=\ell_0}^{T}e^{Z_\ell}$ (with $\ell_0\in\{0,1\}$), cf.\ \cite{Molchan1999,Aurzada,BFFN,AurzadaBaumgartenjpa,castellguillotinwatbled}. In \cite{Molchan1999} it is shown that the continuous-time analog of this quantity behaves as $c T^{H-1}$ for any {\it continuous-time} $H$-self-similar process with stationary increments and a certain other time-reversibility property. Further, certain moment conditions are assumed in \cite{Molchan1999} (also see \cite{Molchan1999preprint}). We will apply the following lemma which does not have these moment conditions and in which $H$-self-similarity
(which does not make sense in discrete time) is replaced by (\ref{ass1}) extracting the "natural scaling'' of the process $Z$.
\begin{lem}\label{Mol}[see Lemma 5 in \cite{AGP}] Let $Z=(Z_n)_{n\in\NN}$ be a stochastic process with 
\begin{equation} \label{ass1}
\lim_{T\to+\infty} \frac{1}{T^{H}\ell(T)} \EE\left[ \sup_{  t\in[0,1]} Z_{[t T]} \right] = \kappa,
\end{equation}
for some $H\in(0,1)$, $\kappa\in(0,\infty)$, and with $\ell$ being a slowly varying function at infinity. Further assume that $Z$ is time-reversible in the sense that for any $T\in\NN$, the vectors $(Z_{T-k}-Z_T)_{k=0,\ldots,T}$ and $(Z_k)_{k=0,\ldots,T}$ have the same law. 
Then, 
$$\limsup_{x\rightarrow +\infty} \frac{x^{1-H}}{ \ell(x)} \EE\Big[ \Big(\displaystyle\sum_{l=0}^{[x]} e^{Z_l}\Big)^{-1} \Big]\leq \kappa H$$
and
$$\liminf_{x\rightarrow +\infty} \frac{x^{1-H}}{\ell(x)} \EE\Big[ \Big(\displaystyle\sum_{l=1}^{[x]} e^{Z_l}\Big)^{-1} \Big]\geq \kappa H.$$
\end{lem}
Note the difference in the summation $l=0,\ldots$ vs.\ $l=1,\ldots$, which complicates the use of this lemma to prove the lower bounds in
Theorems \ref{theoMk} and \ref{theoMk0}. Our additional assumptions for the lower bounds of Theorems  \ref{theoMk} and \ref{theoMk0} come from the fact that
the sum starts from 1 in the second inequality of Lemma \ref{Mol}.

\subsection{Verification of Lemma~\ref{Mol} for RWRS}
The goal of this subsection is to verify that Lemma~\ref{Mol} holds with $H:=3/4$ and $\ell\equiv 1$ for the RWRS $Z$ and for $\tilde Z$. 

We first show that RWRS is time-reversible. Note that
$$
Z_{T-k}-Z_T = \sum_{j=1}^{T-k} \xi_{S_j} -  \sum_{j=1}^{T} \xi_{S_j} =  -  \sum_{j=T-k+1}^{T} \xi_{S_j}, \qquad k=0,\ldots, T.
$$
By conditioning on the random walk and using the symmetry of the environment as well as the fact that the environment is i.i.d.\ (and thus spatially homegeneous), the above vector has the same distribution as
$$
 \sum_{j=T-k+1}^{T} \xi_{S_j-S_{T+1}} =  \sum_{j=1}^{k} \xi_{S_{T-j+1}-S_{T+1}}, \qquad k=0,\ldots, T.
$$
Since $(\xi_y)_y$ and $(\xi_{-y})_y$ have the same distribution, the above vector has the same distribution as
\begin{equation}
\sum_{j=1}^{k} \xi_{S_{T+1}-S_{T-j+1}}, \qquad k=0,\ldots, T. \label{eqn:refreqtimereversible}
\end{equation}
Now we condition on the environment and use that
$$
S_{T+1}-S_{T-j+1}=\sum_{i=T-j+2}^{T+1} X_i=\sum_{i=1}^{j} X_{T+2-i} ,\qquad j=0,\ldots, T,
$$
has the same law as $(S_j)_{j=0,\ldots, T}$, which in connection with (\ref{eqn:refreqtimereversible}) shows the claim that $Z$ is time-reversible.
 
Under the assumptions of Theorem \ref{theoMk0}, using the fact that that $\tilde Z_k=\sum_{\ell=1}^k\xi_{S_{\ell-1}}\varepsilon_\ell=\sum_{\ell=1}^k\xi_{S_\ell}\varepsilon_\ell$ for every positive integer $T$, the vector
$$(\tilde Z_{T-k}-\tilde Z_T)_{k=0,...,T}=\left(-\sum_{\ell=T-k+1}^T\xi_{S_{\ell-1}}\varepsilon_\ell\right) $$
has the same distribution as
$$\left(\sum_{\ell=T-k+1}^T\xi_{S_T-S_{\ell-1}}\varepsilon_\ell\right)_{k=0,...,T} =
\left(\sum_{\ell=T-k+1}^T\xi_{\tilde X_\ell(1-\varepsilon_\ell)+...+\tilde X_T(1-\varepsilon_T)}\varepsilon_\ell\right)_{k=0,...,T}$$
(since $(\xi_y)_y$ and $(-\xi_{S_T-y})_y$ have the same distribution
given $(S,(\varepsilon_k)_k)$),
which has the same distribution as
$$\left(\sum_{\ell=T-k+1}^T\xi_{\tilde X_{T-\ell+1}(1-\varepsilon_{T-\ell+1})+...+\tilde X_1(1-\varepsilon_1)}\varepsilon_{T-\ell+1}\right)_{k=0,...,T}=
\left(\sum_{\ell=1}^k\xi_{S_ \ell}\varepsilon_{\ell}\right)_{k=0,...,T}=\left(\tilde Z_k\right)_{k=0,...,T} $$
(since $(\tilde X_\ell,\varepsilon_\ell)_{\ell=1,...,T}$ and
$(\tilde X_{T-\ell+1},\varepsilon_{T-\ell+1})_{\ell=1,...,T}$
have the same distribution given $\xi$). Hence we have proved the
time-reversibility of $\tilde Z$.

Now let us verify (\ref{ass1}).
Note that the sequence of random variables $T^{-3/4}\max_{k=1,\ldots,T} Z_k$ is uniformly bounded in $L^2$: Indeed, given $S$, 
the random variable $Z_n$ is a sum of associated random variables with zero mean and finite variance, so from Theorem~2 in \cite{NW}, $$\EE[ (\max_{k=1,\ldots, T} Z_k)^2 | S] \leq \EE [ Z_{T}^2| S] = V_{T}.$$
By integrating with respect to the random walk, we get $$\EE[ (\max_{k=1,\ldots, T} Z_k)^2] \leq \EE [ V_{T}] \sim C T^{3/2},$$
cf.\ (2.13) in \cite{KS}.
Since the sequence of processes  $( Z_{[tT]} / T^{3/4} )_{t\geq 0}$ weakly converges for the Skorokhod topology to the process $(\Delta_t)_{t\geq 0}$ (see \cite{KS} and the remark following Theorem 2 of \cite{TLL}), we get 
$$ \lim_{T\rightarrow +\infty}\EE\Big[ \sup_{t\in [0,1]} \Big(\frac{Z_{[tT]}}{T^{3/4}}\Big)\Big] =
\EE\Big[ \sup_{t\in [0,1]} \Delta_t\Big] =:\kappa,$$
which is known to be finite using Proposition 2.1 in \cite{BFFN}.

Under assumptions of Theorem \ref{theoMk0}, we proceed
analogously to prove that (\ref{ass1}) holds for $\tilde Z$.
We obtain 
$$\EE[ (\max_{k=1,\ldots, T} \tilde Z_k)^2] \leq \EE [ \tilde V_{T}] \le\EE[V_T]\sim C T^{3/2},$$
cf.\ (2.13) in \cite{KS}.
The fact that the sequence of processes  $(\tilde{Z}_{[tT]} / T^{3/4} )_{t\geq 0}$ weakly converges for the Skorokhod topology to the process $(K_{\delta} \Delta_t)_{t\geq 0}$ where $K_{\delta} = \frac{\delta}{(1-\delta)^{1/4}}$ has been proved in \cite{GPN} and so 
$$ \lim_{T\rightarrow +\infty}\EE\Big[ \sup_{t\in [0,1]} \Big(\frac{\tilde Z_{[tT]}}{T^{3/4}}\Big)\Big] =K_{\delta}
\EE\Big[ \sup_{t\in [0,1]} \Delta_t\Big] =:\tilde\kappa.$$

\subsection{Proof of the upper bound} 

As in \cite{Molchan1999} and \cite{Aurzada}, the main idea in the proof of the upper
bound in \eqref{eqMk}, is to bound the exponential functionals $\left(\sum_{k=0}^Te^{Z_k}\right)^{-1}$ of Lemma~\ref{Mol}
 from below 
by restricting the expectation  to a well-chosen set of paths.

\noindent Conditionally on $S$, $Z_k$
is the sum of centered and positively associated random variables. 
It follows that for every $0\le u<v<w$ and all real numbers $a,b$, 
\begin{equation}\label{slepian1}
\PP\left[\max_{k=u,\ldots,v} Z_k\le a,\ \max_{k=v+1,\ldots,w}Z_k\le b\Big|  S \right]\ge  
    \PP\left[\max_{k=u,\ldots,v}Z_k\le a\Big|  S \right]
    \PP\left[\max_{k=v+1,\ldots,w}Z_k\le b\Big|  S \right]
\end{equation}
\begin{equation}\label{slepian2}
\PP\left[\max_{k=u,\ldots,v}Z_k\le a,\ \max_{k=v+1,\ldots,w}(Z_k-Z_v)\le b \Big|  S \right]\ge  
    \PP\left[\max_{k=u,\ldots,v}Z_k\le a \Big|  S \right]\PP\left[\max_{k=v+1,\ldots,w}(Z_k-Z_v)\le b\Big|  S\right],
\end{equation}
and analogously, under the assumptions of Theorem \ref{theoMk0},
\begin{multline}\label{slepian10}
\PP\left[\max_{k=u,\ldots,v} \tilde Z_k\le a,\ \max_{k=v+1,\ldots,w}\tilde Z_k\le b\Big|  (S,(\varepsilon_k)_k) \right]\ge  
    \PP\left[\max_{k=u,\ldots,v}\tilde Z_k\le a\Big|  (S,(\varepsilon_k)_k) \right]\\
    \PP\left[\max_{k=v+1,\ldots,w}\tilde Z_k\le b\Big|  (S,(\varepsilon_k)_k) \right]
\end{multline}
\begin{multline}\label{slepian20}
\PP\left[\max_{k=u,\ldots,v}\tilde Z_k\le a,\ \max_{k=v+1,\ldots,w}(\tilde Z_k-\tilde Z_v)\le b \Big| (S,(\varepsilon_k)_k) \right]\ge  
    \PP\left[\max_{k=u,\ldots,v}\tilde Z_k\le a \Big|  (S,(\varepsilon_k)_k) \right]\\
\PP\left[\max_{k=v+1,\ldots,w}(\tilde Z_k-\tilde Z_v)\le b\Big|  (S,(\varepsilon_k)_k)\right].
\end{multline}
Let us precise that these inequalities will play the role of the Slepian Lemma in \cite{BFFN,castellguillotinwatbled,AGP}.
Let $a_T\ge (\log T)^6$ and 
set $\beta_T:=\sigma_{\xi}\sqrt{V_{a_T}}$.
Let us define the random function 
$$\phi(k):= \left\{\begin{array}{ll}
  1  & \mbox{ for } 0 \leq k < a_T \, , \\
1-\beta_T  &  \mbox{ for } a_T\leq k \leq T \, ,
\end{array} 
\right.
$$ 
which is $S$-measurable.
Clearly, we have 
\begin{equation} \label{eqn:argumentphi1}
\EE\left[\left( \sum_{k=0}^T e^{Z_k}\right)^{-1}\Big| S \right] \geq 
\left( \sum_{k=0}^T e^{\phi(k)} \right)^{-1}  \PP\Big[\forall k\in\{0,\ldots,T\}, Z_k\leq \phi(k)\Big|  S \Big].
\end{equation}
From (\ref{slepian1}), we have
\[
 \PP\Big[\forall k\in\{0,\ldots,T\}, Z_k\leq \phi(k)\Big|  S \Big]
 \geq   \PP\Big[\max_{k=0,\ldots,a_T} Z_k\leq 1\Big|  S  \Big] \,  \PP\Big[\max_{ k=a_T,\ldots,T} 
Z_k\leq 1-\beta_T\Big|  S \Big].
\]
Note that
\begin{eqnarray*}
\PP\Big[\max_{k=a_T,\ldots,T} Z_k\leq 1-\beta_T \Big|  S \Big] 
& \geq & \PP\Big[Z_{a_T} 
\leq -\beta_T ; \max_{k=a_T,\ldots,T} (Z_k- Z_{a_T}) \leq 1\Big|  S \Big]
\\
& \geq & \PP\Big[Z_{a_T} \leq -\beta_T\Big| S \Big] \cdot \PP\Big[\max_{k=a_T,\ldots,T} 
(Z_k - Z_{a_T} )\leq 1\Big|  S \Big],
\end{eqnarray*}
by (\ref{slepian2}). 
Moreover, it is easy to check that for every $T>1$
$$\sum_{k=0}^T e^{\phi(k)}  \leq e  (a_T+1+T e^{-\beta_T}).$$
In the following, $C$ is a constant whose value may change but does not depend on $T$. 
Then, summing up (\ref{eqn:argumentphi1}) and the succeeding estimates, we can write that for $T$ large enough
\begin{equation}\label{firststep}
(a_T+Te^ {-\beta_T})\EE\left[\left(\sum_{k=0}^T e^{Z_k} \right)^{-1}\Big| S \right]
\geq C \PP\left[Z_{a_T} \leq -\beta_T| S\right]\,
\PP\big[\max_{k=0,\ldots,a_T}Z_k\leq 1\big| S \big]\,
\PP\big[\max_{k=a_T,\ldots,T}(Z_k-Z_{a_T})\leq 1 \big| S \big].
\end{equation}
The two first probabilities in the right hand side of (\ref{firststep}) can be approximated with the distribution function of the standard Gaussian law $\mathcal{N}(0,1)$. The error by doing this approximation will be controlled by using Lemma~\ref{berry}.
Indeed, we have 
\begin{eqnarray}\label{morceau}
\PP\left[Z_{a_T} \leq -\beta_T| S\right] &=& \PP\left[Z_{a_T} \leq -\beta_T| S\right] - \Phi(-1) +\Phi(-1)\\
&\ge& \Phi(-1)-\tilde C \frac{\EE[|\xi_0|^3]}{\EE[|\xi_0|^2]^{3/2}} \frac{Q_{a_T}^{(3)} }{ V_{a_T}^{3/2}}\nonumber\\
&\ge& \Phi(-1)-\tilde C  \frac{V_{a_T} N_{a_T}^*}{ V_{a_T}^{3/2}}\nonumber\\
&\ge& \Phi(-1)-\tilde C  \frac{ N_{a_T}^*}{ \sqrt{V_{a_T}}}
,\label{morceaubis}
\end{eqnarray}
Moreover since the law of the random scenery is symmetric, from L\'evy's inequality (see for instance Theorem 2.13.1 in \cite{Sto}), we get
\begin{equation}\label{morceau0}
\PP\big[\max_{k=0,\ldots,a_T}Z_k\leq 1\big| S\big]
=1-\PP\big[\max_{k=0,\ldots,a_T}Z_k > 1\big| S \big]
\geq 1-2\PP\big[Z_{a_T}> 1\big| S \big] = \PP\big[\big|Z_{a_T}\big|\leq 1\big| S \big].
  \end{equation}
Now, let $\gamma\in(0,1/48)$, due to Lemma \ref{TLLcond3}, for $T$ large enough such that $a_T\in d{}\mathbb N$,
\begin{eqnarray}
\PP\big[\big|Z_{a_T}\big|\leq 1 |S \big]&\ge& 
\PP\big[Z_{a_T}= 0 |S \big]\nonumber\\
&\ge&\frac{c}{\sqrt{V_{a_T}}},\label{morceau01}
\end{eqnarray} 
holds a.s. on a sequence of $S$-measurable sets $\Omega_{a_T}^{(0)}(\gamma)$.
Due to \eqref{firststep}, \eqref{morceaubis}, \eqref{morceau01}, on $\Omega_{a_T}^{(0)}(\gamma)$, we have
\begin{eqnarray*}
\PP\left[\left . \max_{k=a_T,\ldots,T}(Z_k-Z_{a_T})\leq 1\right|S  \right]
&\le&C\frac{(a_T+Te^ {-\beta_T})\EE\left[\left(\sum_{k=0}^T e^{Z_k} \right)^{-1}\Big| S \right]}
{ \PP\left[Z_{a_T} \leq -\beta_T| S\right]\,
\PP\big[\max_{k=0,\ldots,a_T}Z_k\leq 1\big| S \big]}\\
&\le&C\frac{a_T^{\frac 34+\frac {\gamma}2}(a_T+Te^ {-\sigma_{\xi}a_T^{\frac 34-\frac \gamma 2}})\EE\left[\left(\sum_{k=0}^T e^{Z_k} \right)^{-1}\Big| S \right]}
{(\Phi(-1)-\tilde C  a_T^{-\frac 14+\frac {3\gamma} 2})}\\
&\le& \tilde C_\gamma {a_T^{\frac 34+\frac {\gamma}2}(a_T+Te^ {-\sigma_{\xi}a_T^{\frac 34-\frac \gamma{2}}})\EE\left[\left(\sum_{k=0}^T e^{Z_k} \right)^{-1}\Big| S \right]}
\end{eqnarray*}
for $T$ large enough, 
where we used the facts that $\beta_T=\sigma_{\xi}\sqrt{V_{a_T}}\ge\sigma_{\xi} a_T^{\frac 34-\frac \gamma{2}}$ and that $\frac{ N_{a_T}^*}{ \sqrt{V_{a_T}}}\le a_T^{-\frac 14+\frac{3\gamma} 2}$.
It comes
\begin{equation}\label{limsupfinal}
\PP\left[\max_{k=a_T,\ldots,T}(Z_k-Z_{a_T})\leq 1 \right]\le \mathbb P((\Omega_{a_T}^{(0)}(\gamma))^{c})+
 \tilde C_\gamma {a_T^{\frac 34+\frac {\gamma}2}}(a_T+Te^ {-\sigma_{\xi}a_T^{\frac 34-\frac \gamma{2}}})
\EE\left[\left(\sum_{k=0}^T e^{Z_k} \right)^{-1} \right]
\end{equation}
if $a_T\in d{}\mathbb N$.
Let $\delta_0\in(0,\frac\gamma 2)$, we take $a_T=d{}\left\lceil\frac 1{d{}}((\log T)/4)^{\frac 2{\delta_0}}\right\rceil$, so that $\mathbb P((\Omega_{a_T}^{(0)}(\gamma))^{c})=o\left( e^{-(a_T)^{\frac{\delta_0} 2}}\right)\le T^{-\frac 14}$.
We observe that $a_T^{\frac 34-\frac\gamma 2}\ge c_0 (\log T)^{(\frac 3{\gamma}-2)}>c_0(\log T)^{142}$.
Applying Lemma~\ref{Mol}, we conclude that there exists some constant $c>0$ such that
\begin{equation}\label{Pmax3}
\PP\big[\max_{k=a_T,\ldots,T} (Z_k-Z_{a_T})\leq 1\big] =O\left( (\log T)^{c}\ T^{-\frac{1}{4}}\right).
\end{equation}
The left hand side of \eqref{Pmax3} is greater than the quantity we want to bound from above, since by stationarity of increments,
\begin{equation} \label{eqn:obis}
\PP\big[\max_{k=a_T,\ldots,T}(Z_k-Z_{a_T})\leq 1\big]
=\PP\big[\max_{k=0,\ldots,T-a_T}Z_ {k}\leq 1\big]
\geq \PP\big[\max_{k=0,\ldots,T}Z_ {k}\leq 1\big].
\end{equation}

Let us make now the assumptions of Theorem \ref{theoMk0}.
Analogously, on $\tilde \Omega_{a_T}^{(0)}(\gamma)$, if $\sum_{k=0}^{a_T}\varepsilon_k\in d{}\mathbb N$, the following inequality holds
\begin{equation}
\PP\left[\left . \max_{k=a_T,\ldots,T}(\tilde Z_k-\tilde Z_{a_T})\leq 1\right|
(S,(\varepsilon_k)_k) \right]
\le \tilde C_\gamma {a_T^{\frac 34+\frac {\gamma}2}(a_T+Te^ {-\sigma_{\xi}a_T^{\frac 34-\frac \gamma{2}}})\EE\left[\left(\sum_{k=0}^T e^{\tilde Z_k} \right)^{-1}\Big| (S,(\varepsilon_k)_k) \right]}.
\end{equation}
We take $0<2\delta_0<\tilde\gamma<\gamma<\frac 1{48}$
and $\tilde a_T:=\left\lceil ((\log T)/4)^{\frac 2{\delta_0}}\right\rceil$.
We define $ a_T:=\min\{k\ge \tilde a_T:\sum_{\ell=1}^k\varepsilon_\ell\in d{}\mathbb N\}$ (here $a_T$ is a $(S,(\varepsilon_k)_k)$-measurable random variable).
Since $d{}\le 2$, we observe that
$$\mathbb P(a_T-\tilde a_T>\tilde a_T^{\delta_0})\le\mathbb P\left(\sum_{\ell=1}^{\tilde a_T^{\delta_0}}\varepsilon_\ell=0\right) 
   =(1-\delta)^{\tilde a_T^{\delta_0}}=o\left(e^{-(\tilde a_T)^{\frac{\delta_0}2}}\right)
=o(T^{-\frac 14}).$$
Moreover, there exists $\tilde T_0$ such that, for every $T\ge \tilde T_0$, $\tilde\Omega_{\tilde a_T}^{(0)}(\tilde \gamma)\cap\{a_T\le \tilde a_T+\tilde a_T^{\delta_0}\}\subseteq\tilde\Omega_{ a_T}^{(0)}(\gamma)$.
Hence
\begin{multline}\label{limsupfinal0}
\PP\left[\max_{k=a_T,\ldots,T}(\tilde Z_k-\tilde Z_{a_T})\leq 1 \right]\le \mathbb P((\tilde\Omega_{\tilde a_T}^{(0)}(\tilde\gamma))^{c})+\mathbb P(a_T-\tilde a_T>\tilde a_T^{\delta_0})\\
+ \tilde C_\gamma {a_T^{\frac 34+\frac {\gamma}2}}(a_T+Te^ {- \sigma_{\xi} a_T^{\frac 34-\frac \gamma{2}}})
\EE\left[\left(\sum_{k=0}^T e^{\tilde Z_k} \right)^{-1} \right]
\end{multline}
and we conclude as above.
\subsection{Proof of the lower bound}
Fix $\beta>1/4$ and define $Z_T^*:=\max_{k=1,\ldots,T} Z_k$. Observe that
\begin{eqnarray} \label{eqn:argumentlowerboundalpha}
 \EE \Big[ \Big(\sum_{k=1}^T e^{Z_k} \Big)^{-1} \Big] &= & \EE \Big[ \Big(\sum_{k=1}^T e^{Z_k} \Big)^{-1}\ind_{ Z_T^* \geq \beta \log T} \Big] + 
 \EE \Big[ \Big(\sum_{k=1}^T e^{Z_k} \Big)^{-1}\ind_{ Z_T^* < \beta \log T} \Big] \\
&= : & I_1(T) + I_2(T).
\end{eqnarray}
First, we clearly have 
\begin{eqnarray}
I_1(T) &\leq &\EE [ e^{-Z_T^*} \ind_{ Z_T^* \geq \beta \log T}]  \notag
\\
&\leq & T^{-\beta}. \notag
\end{eqnarray}
We observe that
$$I_2(T)\le \mathbb E\left[e^{-Z_1}\ind_{ Z_T^* < \beta \log T}\right].$$
Let us fix a parameter $\theta\in(0,1)$ and let us define the event $A:=\{ Z_1 \geq  -\log H^{-1}\left(\frac {3\kappa}4\theta T^{-\frac 14}\right) \}$. Then, 
\begin{equation}\label{eqn:summation}
I_2(T) \leq H^{-1}\left(\frac {3\kappa}4\theta T^{-\frac 14}\right) \PP\left[ Z_T^* < \beta \log T \right] + \EE \left[ e^{-Z_1} \ind_{A^{c}} \right]. 
\end{equation}
Since $Z_1$ has the same distribution as $\xi_1$, its distribution is symmetric and so
$$
\EE \left[ e^{-Z_1} \ind_{A^{c}} \right]=H\left(H^{-1}\left(\frac {3\kappa}4\theta T^{-\frac 14}\right)\right)\le\frac {3\kappa}4\theta T^{-\frac 14}.
$$
But 
$$\frac{3\kappa}4\theta<\liminf_{x\rightarrow+\infty}
x^{\frac 14}(I_1(x)+I_2(x))$$
So we have shown that for $T$ large,
\begin{equation}\label{ZN}
\PP\left[Z_T^* \leq \beta \log T\right] \geq c^{-1} T^{-1/4} \left[H^{-1}\left(\frac {3\kappa}4\theta T^{-\frac 14}\right)\right]^{-1}.
\end{equation}
Let $\gamma\in(0,1/48)$, $\delta_0\in(0,\frac\gamma 2)$ and $a_T= d{}\lceil (\beta \log T)^{2/\delta_0}/d{}\rceil$. Note that from inequalities (\ref{slepian1}) and (\ref{slepian2}), we have
\begin{eqnarray} 
\PP\left[\left. \max_{k=1,\ldots,T} Z_k\leq 1 \right|S\right] &\geq& \PP\left[ \max_{k=1,\ldots,a_T} Z_k\leq 1; Z_{a_T}\leq - \beta \log T;\notag \right.\\ && \qquad\left.\left. \max_{k=a_T+1,\ldots, T} Z_k - Z_{a_T} \leq \beta \log T\right|S\right] \notag \\
&\geq & \PP\left[ \left. \max_{k=1,\ldots,a_T} Z_k\leq 1\right|S\right] \cdot \PP\left[Z_{a_T}\leq - \beta \log T|S\right] \notag \\ && \qquad \cdot \PP\left[ \left.\max_{k=a_T+1,\ldots, T} Z_k - Z_{a_T} \leq \beta \log T\right|S\right]\label{eqn:logtermslepian}
\end{eqnarray}
From Lemma \ref{TLLcond3}, (\ref{morceau0}), (\ref{morceau01}) and Lemma \ref{berry}, for $T$ large enough, on ${\Omega}_{a_T}^{(0)}(\gamma)$, 
\begin{eqnarray}
&\ &\ \ \PP\left[ \left.\max_{k=a_T+1,\ldots, T} Z_k - Z_{a_T} \leq \beta \log T\right|S\right] \\
&\leq&  c \sqrt{V_{a_T}} \left(\Phi\Big( -\frac{\beta \log T}{\sigma_{\xi}\sqrt{V_{a_T}}}\Big) - \tilde{C}  \frac{ N_{a_T}^*}{ \sqrt{V_{a_T}}}\right)^{-1} \PP\left[\left. \max_{k=1,\ldots,T} Z_k\leq 1 \right|S\right]\notag \\
&\leq & c a_T^{\frac{3+2\gamma}{4}}\left(\Phi\Big( -\frac{\beta\log T}{\sigma_{\xi}a_T^{\frac{3-2\gamma}{4}}}\Big)- \tilde{C} a_T^{-\frac{1-6\gamma}{4}}\right)^{-1}
\PP\left[\left. \max_{k=1,\ldots,T} Z_k\leq 1 \right|S\right]\notag \\ 
&\leq & c a_T^{\frac{3+2\gamma}{4}}\left(\Phi(-1)- \tilde{C} a_T^{-\frac{1-6\gamma}{4}}\right)^{-1}
\PP\left[\left. \max_{k=1,\ldots,T} Z_k\leq 1 \right|S\right] \notag \\
&\leq & c (\log T)^{\frac{3+2\gamma}{2\delta_0} } \PP\left[\left. \max_{k=1,\ldots,T} Z_k\leq 1 \right|S\right]\label{eqn:gaussiandirect}
\end{eqnarray}
Thus, using the last inequality and the stationarity of the increments, we get 
\begin{eqnarray} \label{ZN2}
\PP\left[    Z_T^* \leq \beta \log T\right]
&\leq & \EE \left[\PP\left[\left. \max_{k=a_T+1,\ldots, T} Z_k - Z_{a_T} \leq \beta \log T\right|S\right]\right] \nonumber\\
&\leq &  \EE\left[\PP\left[\left. \max_{k=a_T+1,\ldots, T} Z_k - Z_{a_T} \leq \beta \log T\right|S\right] {\bf 1}_{{\Omega}_{a_T}^{(0)}(\gamma)} \right] + \PP[({\Omega}_{a_T}^{(0)}(\gamma))^c] \nonumber\\
&\leq & c (\log T)^{\frac{3+2\gamma}{2\delta_0} } \PP\left[ \max_{k=1,\ldots,T} Z_k\leq 1 \right] + \PP[({\Omega}^{(0)}_{a_T}(\gamma))^c].
\end{eqnarray}
Since $\mathbb P((\Omega_{a_T}^{(0)}(\gamma))^{c})=\mathcal O\left( e^{-(a_T)^{\frac{\delta_0}2}}\right)=\mathcal O(T^{-\beta})$, by combining (\ref{ZN}) and (\ref{ZN2}), we get the lower bound.

Under the assumptions of Theorem \ref{theoMk0}, we proceed analogously by replacing $Z$ by $\tilde Z$ (and $V$ by $\tilde V$) and we obtain, for $T$ large enough,
\begin{equation}\label{ZN0}
\PP\left[\tilde Z_{T-a_T}^* \leq \beta \log T\right]\geq
\PP\left[\tilde Z_T^* \leq \beta \log T\right] \geq c^{-1} T^{-1/4} \left[\tilde H^{-1}\left(\frac {3\kappa}4\theta T^{-\frac 14}\right)\right]^{-1}.
\end{equation}
where $\tilde H$ is given by $\tilde H(t):=\mathbb E[e^{\xi_1\varepsilon_1}{\mathbf 1}_{\{e^{\xi_1\varepsilon_1}>t\}}]=\delta H(t)+(1-\delta)\ind_{\{t<1\}}$ (hence $\tilde H^{-1}(u)=H^{-1}(u/\delta)$ as soon as $u<\delta H(1)$)
and
\begin{equation}
\PP\left[ \left.\max_{k=a_T+1,\ldots, T} \tilde Z_k -\tilde Z_{a_T} \leq \beta \log T\right|(S,(\varepsilon_k)_k)\right] \le c (\log T)^{\frac{3+2\gamma}{2\delta_0} } \PP\left[\left.  \tilde Z_T^*\leq 1 \right|(S,(\varepsilon_k)_k)\right]\label{eqn:gaussiandirect0}
\end{equation}
on $\tilde \Omega^{(0)}_{a_T}$ provided $\sum_{k=1}^{a_T}\varepsilon_k\in d{}\mathbb N$.
We proceed now as for the upper bound.
We take $
\tilde a_T:=\left\lceil ((\log T)/4)^{\frac 2{\delta_0}}\right\rceil$
and $0<2\delta_0<\tilde\gamma<\gamma<\frac 1{48}$
and $\tilde a_T= \lceil (\beta \log T)^{2/\delta_0}\rceil$.
We define again $a_T:=\min\{k\ge \tilde{a}_T:\sum_{\ell=1}^k\varepsilon_\ell\in d{}\mathbb N\}$.
Using the stationarity of $(\tilde Z_k)_k$, we obtain
\begin{eqnarray*}
c^{-1} T^{-1/4}  \left[H^{-1}\left(\frac {3\kappa}{4\delta}\theta T^{-\frac 14}\right)\right]^{-1}&\le&
\PP\left[\tilde Z_{T-a_T}^*\leq 1 \right]\\
&\le&
\mathbb E\left[\PP\left[ \left.\max_{k=a_T+1,\ldots, T} \tilde Z_k - \tilde Z_{a_T} \leq \beta \log T\right|(S,(\varepsilon_k)_k)\right]\right]\\
&\le&   \mathbb P((\tilde \Omega_{\tilde a_T}^{(0)}(\tilde\gamma))^{c})+\mathbb P(a_T-\tilde a_T>\tilde a_T^{\delta_0})\\
&\ &+ c (\log T)^{\frac{3+2\gamma}{2\delta_0} } \PP\left[ \tilde Z_T^*\leq 1 \right]
\end{eqnarray*}
for $T$ large enough, from which we conclude.


\begin{thebibliography}{00}
\bibitem{Aurzada} Aurzada, F. {\it On the one-sided exit problem for fractional Brownian motion.}
Electron. Commun. Probab., 16:392--404, 2011.

\bibitem{AurzadaBaumgartenjpa} Aurzada, F.; Baumgarten, C. {\it Persistence of fractional Brownian motion with moving boundaries and applications.}
Journal of Physics A: Mathematical and Theoretical 46 (2013), 125007.

\bibitem{AGP} Aurzada, F.; Guillotin-Plantard, N. {\it Persistence exponent for discrete-time, time-reversible processes.} Submitted. 

\bibitem{AS} Aurzada, F.; Simon, T.  {\it Persistence probabilities \& exponents.} To appear in: L\'evy matters, Springer, arXiv:1203.6554, 2012.

\bibitem{Bo89}
Bolthausen, E. {\it A central limit theorem for two-dimensional random walks in random sceneries.}
Ann.\ Probab.\ 17 (1989) 108--115.



\bibitem{Borodin} Borodin, A. N. 
{\it A limit theorem for sums of independent random variables defined 
on a recurrent random walk.}
(Russian)  Dokl. Akad. Nauk SSSR 246(4):786--787, 1979. 

\bibitem{BMS13}
Bray, A.~J.; Majumdar, S.~N.; and Schehr, G.
\newblock {\it Persistence and first-passage properties in non-equilibrium systems.}
\newblock  Advances in Physics, 62(3):225--361, 2013.

\bibitem{CP} Campanino, M. and P\'etritis, D. {\it Random walks on randomly oriented lattices.}  Mark. Proc. Relat. Fields (2003), 9, 391--412.


\bibitem{TLL} Castell, F.; Guillotin-Plantard, N.; P\`ene, F.; and Schapira, B. 
{\it A local limit theorem for random walks in random scenery and on randomly oriented lattices}.  Annals of Probability 39 (6), 2079--2118, 2011.

\bibitem{BFFN} Castell, F.; Guillotin-Plantard, N.; P\`ene, F.; and Schapira, B. 
{\it On the one-sided exit problem for stable processes in random scenery.} 
Electron. Commun. Probab. 18(33):1--7, 2013.

\bibitem{castellguillotinwatbled} Castell, F.; Guillotin-Plantard, N.; and Watbled, F. {\it Persistence exponent for random processes in Brownian scenery}. Preprint, \verb+https://hal.archives-ouvertes.fr/hal-01017142v2+

\bibitem{Fe} Feller, W. An introduction to probability theory and its applications. Vol. II. Second edition, John Wiley and Sons, Inc., New York-London-Sydney, (1971). 

\bibitem{GPN}Guillotin-Plantard, N. and Le Ny, A. {\it  A functional limit theorem for a 2d- random walk with dependent marginals }
Electronic Communications in Probability (2008), Vol. {\bf 13}, 337--351. 

\bibitem{GuPo} Guillotin-Plantard, N. and Poisat, J. {\it  Quenched central limit theorems for random walks in random scenery}.
Stochastic Process. Appl. 123 (4) (2013) 1348--1367.


\bibitem{KS} Kesten, H. and Spitzer, F. 
{\it  A limit theorem related to a new class of self-similar processes.} 
Z. Wahrsch. Verw. Gebiete 50:5--25, 1979.  
%

\bibitem{KL} Khoshnevisan, D. and Lewis, T. M. {\it A law of iterated logarithm for stable processes in random scenery.}
Stochastic Process. Appl., 74(1):89--121, 1998.


\bibitem{Maj1} Majumdar, S. {\it Persistence in nonequilibrium systems}. Current Science 77 (3):370-375, 1999.

\bibitem{Maj}  Majumdar, S. {\it Persistence of a particle in the Matheron - de Marsily velocity field.} 
Phys. Rev. E 68, 050101(R), 2003.

\bibitem{marcusrosen}
Marcus, M. B. and Rosen, J. Markov processes, Gaussian processes, and local times. Cambridge Studies in Advanced Mathematics, 100. Cambridge University Press, Cambridge, 2006.

 \bibitem{MdM} Matheron, G. and de Marsily G.
 {\it Is transport in porous media always diffusive? A counterexample.}
 Water Resources Res. 16:901--907, 1980. 

\bibitem{Molchan1999preprint} Molchan, G.M. {\it Maximum of fractional Brownian motion: probabilities of small values.}
Preprint, \verb+https://www.ma.utexas.edu/mp_arc/c/00/00-195.ps.gz+


\bibitem{Molchan1999} Molchan, G.M. {\it Maximum of fractional Brownian motion: probabilities of small values.}
Comm. Math. Phys., 205(1):97--111, 1999.
%


%
%


\bibitem{NW} Newman, C. M. and Wright, A. L. {\it An invariance principle for certain dependent sequences}. Ann. Probab. 
9, (1981), no. 4, 671-- 675. 


\bibitem{oshaninetal}
Oshanin, G.; Rosso, A.; and Schehr, G. {\it Anomalous Fluctuations of Currents in Sinai-Type Random Chains with Strongly Correlated Disorder}.
Phys. Rev. Lett. 110 (2013), 100602.


\bibitem{Pitt} Pitt, L. {\it Positively correlated normal variables are associated}.  Ann.\ Probab.\ Vol. 10, No 2, (1982) 496 -- 499.

%
%

 \bibitem{samorodnitsky} Samorodnitsky, G. {\it Long range dependence.}
Found. Trends Stoch. Syst. 1 (2006), no. 3, 163--257.

\bibitem{slepian} Slepian, D.
{\it The one-sided barrier problem for Gaussian noise.}
Bell System Tech. J. 41 (1962), 463--501. 

\bibitem{S76}
Spitzer, F. Principles of Random Walks. Second ed., in: Graduate Texts in Mathematics, vol. 34, Springer-Verlag, New-York, 1976.


\bibitem{Sto} Stout, W. Almost sure convergence. (1974) Probability and mathematical statistics. Academin Press.

\bibitem{taqqu} Taqqu, M.S. {\it Weak convergence to fractional Brownian motion and to the Rosenblatt process.} Z. Wahrscheinlichkeitstheorie und Verw. Gebiete 31 (1974/75), 287--302.





%
%
%

\end{thebibliography}
\end{document}